\documentclass{amsart}

\usepackage{tikz-cd}
\usepackage{tikz}
\usepackage{graphicx}
\usepackage{hyperref}
\usepackage{cleveref}
\usepackage{mathrsfs}
\newtheorem{theorem}{Theorem}
\newtheorem{lemma}{Lemma}[section]
\newtheorem{proposition}[lemma]{Proposition}
\newtheorem*{conjecture1}{Conjecture}

\theoremstyle{definition}
\newtheorem{definition}[lemma]{Definition}
\newtheorem{remark}[lemma]{Remark}

\numberwithin{equation}{section}

\begin{document}

\title{Coxeter Systems for which the Brink-Howlett automaton is minimal}

\author{James Parkinson}
\address{School of Mathematics and Statistics, University of Sydney, NSW, Australia}
\email{jamesp@maths.usyd.edu.au}

\author{Yeeka Yau}
\address{School of Mathematics and Statistics, University of Sydney, NSW, Australia}
\email{y.yau@maths.usyd.edu.au}


\date{November 2018}



\begin{abstract}
In this article we classify the Coxeter systems for which the Brink-Howlett automaton is minimal. We show that this automaton is minimal if and only if each elementary root is supported on a standard spherical subsystem, thereby resolving a conjecture of Hohlweg, Nadeau and Williams.
\end{abstract}

\maketitle

\section*{Introduction}
In their celebrated 1993 paper~\cite{Brink_Howlett}, Brink and Howlett proved that all finitely generated Coxeter systems $(W,S)$ are automatic. In particular, they constructed a finite state automaton, which we denote by $\mathcal{A}_{BH}$, recognising the language of reduced words of $(W,S)$. The construction of this automaton utilises the standard geometric representation of Coxeter groups via the associated root system $\Phi$, and the key property discovered by Brink and Howlett is the finiteness of the set $\mathscr{E}$ of \textit{elementary roots}.
\par
In certain cases the automaton $\mathcal{A}_{BH}$ is known to be minimal, in the sense that it has the fewest states amongst all automata recognising the language of reduced words of $(W,S)$. For example Eriksson \cite{Eriksson} and Headley \cite{Headley} independently proved that $\mathcal{A}_{BH}$ is minimal if $W$ is an affine Coxeter group of type $\tilde{A}_n$. However in general $\mathcal{A}_{BH}$ is not minimal, with the simplest example being type $\tilde{C}_2$ in which case the automaton $\mathcal{A}_{BH}$ has one superfluous state.
\par
Recently, Hohlweg, Nadeau and Williams \cite{Hohlweg} introduced a combinatorial approach to constructing an automaton recognising the language of reduced words of $(W,S)$ using the weak order on $(W,S)$ and the notion of a Garside shadow, building on the work of Dehornoy, Dyer and Hohlweg \cite{Dehornoy} and Dyer and Hohlweg \cite{Dyer}. In \cite{Hohlweg}, relations between the Garside shadow automaton $\mathcal{A}_G$ and the Brink-Howlett automaton $\mathcal{A}_{BH}$ are investigated and the following conjecture \cite[Conjecture 2]{Hohlweg} is stated:

\begin{conjecture1} \label{the_conjecture}
The Brink-Howlett automaton $\mathcal{A}_{BH}$ is minimal if and only if $\mathscr{E} = \Phi_{\mathrm{sph}}^+$, where $\Phi_{\mathrm{sph}}^+$ is the set of roots supported on a standard spherical subsystem.
\end{conjecture1}

In \cite{Hohlweg} this conjecture was shown to hold in the following cases:
\begin{enumerate}
    \item when $W$ is finite,
    \item when $W$ is right-angled,
    \item when the Coxeter graph $\Gamma$ is a complete graph,
    \item when $W$ is of type $\tilde{A}_n$,
    \item when $W$ has rank 3.
\end{enumerate}

In this paper we prove the above conjecture. Moreover, we provide a classification of  the Coxeter systems for which $\mathcal{A}_{BH}$ is minimal in terms of excluded subgraphs of the Coxeter graph. Let $\mathscr{X}$ denote the set of connected Coxeter graphs which are either of affine or compact hyperbolic type and contain neither circuits nor infinite bonds. Specifically $\mathscr{X}$ consists of the Coxeter graphs of the irreducible affine Coxeter groups other than type $\tilde{A}_n$, along with the graphs
\medskip

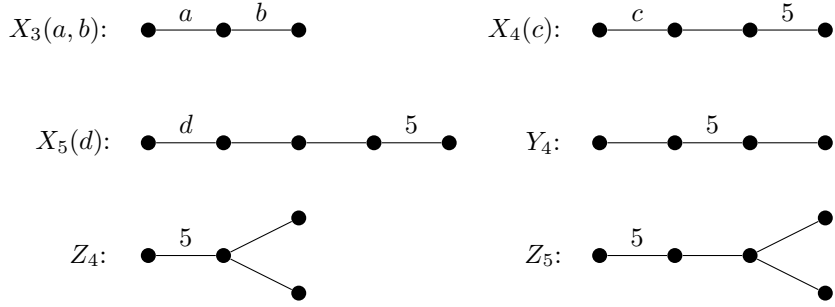
\begin{tikzpicture}[scale=1]
    \node at (0.8,0) {$X_3(a,b)$:}; 
    \draw
    (2,0) node[fill=black,circle,scale=0.6] (0) {}
    (3,0) node[fill=black,circle,scale=0.6] (1)
    {}
    (4,0) node[fill=black,circle,scale=0.6] (2)
    {};
    \draw (0)--node [midway,above]{$a$}(1);
    \draw (1)--node [midway,above]{$b$}(2);
 \node at (7,0) {$X_4(c)$:};
    \draw
    (8,0) node[fill=black,circle,scale=0.6] (0) {}
    (9,0) node[fill=black,circle,scale=0.6] (1)
    {}
    (10,0) node[fill=black,circle,scale=0.6] (2)
    {}
    (11,0) node[fill=black,circle,scale=0.6] (3)
    {};
    \draw (0)--node [midway,above]{$c$}(1);
    \draw (1)--(2);
    \draw (2)--node [midway,above]{$5$}(3);
 \node at (1,-1.5) {$X_5(d)$:};
    \draw
    (2,-1.5) node[fill=black,circle,scale=0.6] (0) {}
    (3,-1.5) node[fill=black,circle,scale=0.6] (1)
    {}
    (4,-1.5) node[fill=black,circle,scale=0.6] (2)
    {}
    (5,-1.5) node[fill=black,circle,scale=0.6] (3)
    {}
    (6,-1.5) node[fill=black,circle,scale=0.6] (4)
    {};
    \draw (0)--node [midway,above]{$d$}(1);
    \draw (1)--(2);
    \draw (2)--(3);
    \draw (3)--node [midway,above]{$5$}(4);
\node at (7.25,-1.5) {$Y_4$:};
    \draw
    (8,-1.5) node[fill=black,circle,scale=0.6] (0) {}
    (9,-1.5) node[fill=black,circle,scale=0.6] (1)
    {}
    (10,-1.5) node[fill=black,circle,scale=0.6] (2)
    {}
    (11,-1.5) node[fill=black,circle,scale=0.6] (3)
    {};
    \draw (0)--(1);
    \draw (1)--node [midway,above]{$5$}(2);
    \draw (2)--(3);
\node at (1.2,-3) {$Z_4$:};
    \draw
    (2,-3) node[fill=black,circle,scale=0.6] (0) {}
    (3,-3) node[fill=black,circle,scale=0.6] (1)
    {}
    (4,-2.5) node[fill=black,circle,scale=0.6] (2)
    {}
    (4,-3.5) node[fill=black,circle,scale=0.6] (3)
    {};
    \draw (0)--node [midway,above]{$5$}(1);
    \draw (1)--(2);
    \draw (1)--(3);
\node at (7.25,-3) {$Z_5$:};
    \draw
    (8,-3) node[fill=black,circle,scale=0.6] (0) {}
    (9,-3) node[fill=black,circle,scale=0.6] (1)
    {}
    (10,-3) node[fill=black,circle,scale=0.6] (2)
    {}
    (11,-2.5) node[fill=black,circle,scale=0.6] (3)
    {}
    (11,-3.5) node[fill=black,circle,scale=0.6] (4)
    {};
    \draw (0)--node [midway,above]{$5$}(1);
    \draw (1)--(2);
    \draw (2)--(3);
    \draw (2)--(4);
\end{tikzpicture}
\medskip

\noindent where $a,b<\infty$ with $\frac{1}{a}+\frac{1}{b}<\frac{1}{2}$, $c\in\{4,5\}$, and $d\in\{3,4,5\}$. Note that if $a\leq b$ then either $(a,b)\in\{(4,5),(5,5)\}$, or $3\leq a<\infty$ and $6\leq b<\infty$ with $(a,b)\neq (3,6)$. 

The main theorem of this paper is the following.

\begin{theorem} \label{Theorem 1}
Let $(W,S)$ be a finitely generated Coxeter system. The following are equivalent:
\begin{enumerate}
    \item[(1)] The Brink-Howlett automaton $\mathcal{A}_{BH}$ is minimal.
    \item[(2)] The Coxeter graph of $(W,S)$ does not have a subgraph contained in~$\mathscr{X}$.
    \item[(3)] The set of elementary roots is $\mathscr{E} = \Phi^{+}_{\mathrm{sph}}$.
\end{enumerate}
\end{theorem}

\section{Preliminaries}

In this section we recall some standard facts regarding Coxeter systems $(W,S)$ from \cite{Combinatorics} and \cite{Humphreys}. We also recall the construction of crystallographic root systems and some general facts regarding finite state automata from \cite{Word_Processing} and \cite{holt_rees}.

\subsection{Coxeter systems and root systems} \label{coxeter_systems_and_root_systems}

Let $(W,S)$ be a Coxeter system with $|S| < \infty$. Let $m(s,t)$ denote the order of $st$ for $s,t \in S$, and let $\Gamma=\Gamma(S)$ denote the Coxeter graph of $(W,S)$. For $J\subseteq S$ let $W_J$ be the standard parabolic subgroup generated by $J$. We say that $J$ is \textit{spherical} if $|W_J| < \infty$. Let $\Gamma(J)$ be the subgraph of $\Gamma$ with vertex set $J$. Thus $\Gamma(J)$ is the Coxeter graph of the standard parabolic Coxeter system $(W_J,J)$. 

The \textit{length} of $w\in W$ is 
$$
\ell(w)=\min\{n\geq 0\mid w=s_1\cdots s_n\text{ with }s_1,\ldots,s_n\in S\},
$$
and an expression $w=s_1\cdots s_n$ with $n=\ell(w)$ is called a \textit{reduced expression} for~$w$. For $J \subseteq S$ spherical, let $w_J$ denote the unique longest element of $W_J$.

The \textit{(left) descent set} of $w\in W$ is 
$$
D(w)=\{s\in S\mid \ell(sw)<\ell(w)\}.
$$ 

Let $V$ be an $\mathbb{R}$-vector space with basis $\{\alpha_s\mid s\in S\}$. Define a symmetric bilinear form on $V$ by linearly extending $\langle\alpha_s,\alpha_t\rangle=-\cos(\pi/m(s,t))$. The Coxeter group $W$ acts on $V$ by the rule $s(v)=v-2\langle v,\alpha_s\rangle \alpha_s$ for $s\in S$ and $v\in V$, and the \textit{root system} of $W$ is $\Phi=\{w(\alpha_s)\mid w\in W,\,s\in S\}$. The elements of $\Phi$ are called \textit{roots}, and the \textit{simple roots} are the roots $\alpha_s$ with $s\in S$. 

Each root $\alpha\in\Phi$ can be written as $\alpha=\sum_{s\in S}c_s\alpha_s$ with either $c_s\geq 0$ for all $s\in S$, or $c_s\leq 0$ for all $s\in S$. In the first case $\alpha$ is called \textit{positive} (written $\alpha>0$), and in the second case $\alpha$ is called \textit{negative} (written $\alpha<0$). Let $\Phi^+$ be the set of all positive roots. The \textit{support} of $\alpha \in \Phi$ is the set $J(\alpha)=\{s\in S\mid c_s\neq 0\}$. For $J \subseteq S$ define $\Phi_J^{+} \subseteq \Phi^{+}$ to be the set of positive roots $\alpha$ with $J(\alpha) = J$. Moreover, we write $\Gamma(\alpha)=\Gamma(J(\alpha))$ for the associated subgraph of $\Gamma$. Let 
$$
\Phi_{\mathrm{sph}}^+=\{\alpha\in \Phi^+\mid J(\alpha)\text{ is spherical}\}.
$$

The \textit{inversion set} of $w\in W$ is 
$$
\Phi(w)=\{\alpha\in \Phi^+\mid w(\alpha)<0\}.
$$

A root $\alpha\in\Phi^+$ is said to \textit{dominate} a root $\beta\in \Phi^+$ if $w(\alpha)<0$ implies that $w(\beta)<0$ (for all $w\in W$). A root $\alpha\in \Phi^+$ is said to be \textit{elementary} if $\alpha$ dominates no other positive root $\beta\neq \alpha$. We note that these roots are also called \textit{small}, \textit{humble} or \textit{minimal} in the literature. 

Let $\mathscr{E}\subseteq \Phi^+$ denote the set of all elementary roots. By \cite[Proposition 2.2(i)]{Brink} we have $\Phi_{\mathrm{sph}}^+\subseteq \mathscr{E}$. The key result of \cite{Brink_Howlett} is that $\mathscr{E}$ is a finite set for all finitely generated Coxeter systems $(W,S)$. 
The \textit{elementary inversion set} of $w\in W$ is 
$$
\mathscr{E}(w)=\{\alpha\in\mathscr{E}\mid w(\alpha)<0\}=\Phi(w)\cap \mathscr{E}.
$$

We recall the following result of Brink \cite{Brink}.

\begin{lemma}{\cite[Lemma 4.1]{Brink}} \label{no_circuits_no_infinitebonds}
Let $\alpha \in \Phi^+$ be such that $\Gamma(\alpha)$ contains a circuit or an infinite bond. Then $\alpha \notin \mathscr{E}$.
\end{lemma}

\subsection{Crystallographic affine root systems} \label{affine_root_system} In the case of affine Coxeter groups there is a useful explicit construction of the root system and the elementary roots.
\par
This construction starts with a reduced, irreducible, \textit{crystallographic} root system $\Phi_0$ in a Euclidean vector space $V_0$ with positive definite inner product $\langle\cdot,\cdot\rangle$ (see, for example, \cite{Bourbaki} or \cite[\S3.3]{Lehrer}). Note that the crystallographic condition allows for roots of different lengths, and hence this is a slight modification of the general setup outlined in \Cref{coxeter_systems_and_root_systems}. The notions of dominance and elementary roots extend verbatim to this setting. 
\par
Let $\{ \alpha_1,\ldots,\alpha_n \}$ be a set of simple roots of $\Phi_0$ (when a choice of indexing is required we will use the Bourbaki conventions~\cite{Bourbaki}). For $\alpha\in\Phi_0$ let $s_{\alpha}$ be the reflection $s_{\alpha}(v)=v-\langle v,\alpha^{\vee}\rangle\alpha$, where $\alpha^{\vee}=2\alpha/\langle\alpha,\alpha\rangle$. Let $(W_0,S_0)$ be the Coxeter system with $S_0=\{s_i\mid 1\leq i\leq n\}$, where we set $s_i=s_{\alpha_i}$.
\par
Now define $ V:= V_0 \oplus \mathbb{R}\delta$. The bilinear form $\langle \cdot, \cdot \rangle$ extends uniquely to a symmetric positive semidefinite bilinear form on $V$ with radical $\mathbb{R}\delta$ by 
\begin{equation*}
    \langle \alpha + l \delta , \beta + k\delta \rangle = \langle \alpha, \beta \rangle \quad\text{for $\alpha,\beta\in V_0$ and $l,k\in\mathbb{R}$}.
\end{equation*}
The affine root system is $\Phi=\Phi_0+\mathbb{Z}\delta$. In particular, the set of positive affine roots is 
\begin{equation*}
\Phi^+=(\Phi_0^+ + \mathbb{Z}_{\geq 0}\delta)\cup(-\Phi_0^+ + \mathbb{Z}_{>0}\delta).
\end{equation*} 
Let $W$ be the subgroup of $GL(V)$ generated by $\{ s_{\alpha + k \delta} \mid \alpha + k \delta \in \Phi \}$. Note that
\begin{equation*}
    s_{\alpha + l \delta} (\beta + k \delta) = s_{\alpha}(\beta) + (l -k\langle \beta, \alpha^{\vee} \rangle) \delta.
\end{equation*}
Define $\alpha_0 = - \varphi + \delta$ where $\varphi$ is the highest root of $\Phi_0$ (see \cite{Lehrer}). Then $(W,S)$ is an affine Coxeter system with $S = \{ s_i \mid 0 \le i \le n \}$ where $s_0 := s_{\alpha_0}$. 
\par
In terms of the construction of the crystallographic root system described here, it is easy to see that the set of elementary roots is
\begin{equation*}
\mathscr{E}=(\Phi_0^+)\cup (-\Phi_0^+ + \delta).
\end{equation*}

\subsection{Automata}

Recall the following definition of automata from~\cite{Word_Processing}.

\begin{definition}
An \textit{automaton} $\mathcal{A}$ is a quintuple $(X, S, \mu, Y, o)$, where $X$ is a set, called the \textit{state set}, $S$ is a finite set called the \textit{alphabet}, $\mu: X \times S \rightarrow X$ is a function, called the \textit{transition function}, $Y \subseteq X $ is the set of \textit{accept} states, and $o$ is the \textit{initial} state. A \textit{finite state automaton} is an automaton $\mathcal{A}=(X,S,\mu,Y,o)$ with $|X|<\infty$.
\end{definition}

Since we are interested in automata recognising the language of reduced words of a Coxeter system, for the automata described in this article every state is an accept state (hence $X=Y$) and the alphabet is the set of Coxeter generators $S$. We view an automaton $\mathcal{A}$ as a directed graph with edges labelled by elements of $S$. The set of vertices of this graph is the set $X$ of states, and there is a directed edge labelled~$s$ from vertex $x$ to vertex $y$ if and only if $\mu(x,s) = y$. A word  $w = s_1 s_2 \cdots s_n$ is reduced if and only if $(s_1, s_2 , \ldots , s_n)$ corresponds to a sequence of directed edges labeled $s_i$ for $1 \le i \le n$ in the automaton $\mathcal{A}$, beginning from the initial state $o$.
\par
The automaton $\mathcal{A}_{BH}$ constructed in \cite{Brink_Howlett} is defined using elementary inversion sets of $W$. Let $\mathscr{E}(W)$ be the set of all elementary inversion sets. Note that $\mathscr{E}(W)$ is a finite set because the set $\mathscr{E}$ is finite (see \cite[Theorem 2.8]{Brink_Howlett}). The automaton $\mathcal{A}_{BH}$ is defined in the following way:

\begin{enumerate}
    \item The set of states is $\mathscr{E}(W)$ with initial state $\mathscr{E}(e)$.
    \item The transition function $\mu: \mathscr{E}(W) \times S \rightarrow \mathscr{E}(W)$ is defined by $\mathscr{E}(w) \xrightarrow{s} \mathscr{E}(ws)$ if $\alpha_s \notin \mathscr{E}(w)$.
\end{enumerate}

See \cite{Brink_Howlett} for complete details regarding the Brink-Howlett automaton and proof of the finiteness of the elementary roots. We note that in fact Brink and Howlett construct an automaton recognising the language of lexicographically minimal reduced words, however the construction given above is implicit in their paper. An exposition is also in \cite[Chapter 4]{Combinatorics}.
\par
We now construct the unique minimal automaton recognising the language of reduced words in $(W,S)$, following \cite[p16 and \S3.2]{Word_Processing}. Note that this construction applies to all finitely generated groups $(G,S)$, however we will focus on the case of Coxeter groups here. 

The \textit{cone type} of $w\in W$ is 
$$
T(w)=\{v\in W\mid \ell(wv)=\ell(w)+\ell(v)\}.
$$
Let $T(W)=\{T(w)\mid w\in W\}$ denote the set of all cone types. Let $${\mathcal{A}_0=(T(W),S,\mu_0,T(W),T(e))}$$ 
where $\mu_0(T,s)=T'$ if and only if there exists $w\in W$ such that $T=T(w)$ and $T'=T(ws)$ with $\ell(ws)=\ell(w)+1$. It is well known, and easy to check, that $\mathcal{A}_0$ an automaton recognising the language of reduced words of $(W,S)$. Moreover $\mathcal{A}_0$ is the unique minimal such automaton in the following sense.

\begin{proposition}[Myhill-Nerode] \label{myhill_nerode} Let $\mathcal{A}=(X,S,\mu,X,o)$ be an automaton recognising the language of reduced words of $(W,S)$. Then there is a unique surjective map $\theta:X\to T(W)$ such that if $\mu(x,s)=y$ then $\mu_0(\theta(x),s)=\theta(y)$. 
\end{proposition} 

\begin{proof}
See \cite[Theorem 1.2.9]{Word_Processing} and the proof thereof.
\end{proof}

In particular, note that the map $\theta$ from the set of states of $\mathcal{A}_{BH}$ to the set of states of $\mathcal{A}_0$ is given by
$$
\theta(\mathscr{E}(w))=T(w).
$$
Moreover, note that the finiteness of the set $\mathscr{E}(W)$ implies that the set $T(W)$ of cone types is finite. Furthermore, since minimality of $\mathcal{A}_{BH}$ is equivalent to injectivity of~$\theta$, we have that $\mathcal{A}_{BH}$ is minimal if and only if $T(w)=T(v)$ whenever $\mathscr{E}(w)=\mathscr{E}(v)$.

\section{Proof of \cref{Theorem 1}}

The proof of \Cref{Theorem 1} relies on the following key lemma, which gives a particular condition under which the Brink-Howlett automaton is not minimal.

\begin{lemma} \label{key_lemma}
Let $(W,S)$ be a finitely generated Coxeter system. If there exists $ J \subset S$ and $t \in S$ such that:
\begin{enumerate}
    \item[(i)] $J$ is spherical, and
    \item[(ii)] $J \cup \{ t \}$ is not spherical, and
    \item[(iii)] $w_J(\alpha_t) \in \mathscr{E}$,
\end{enumerate}
then the automaton $\mathcal{A}_{BH}$ is not minimal.
\end{lemma}

\begin{proof}
Since $w_J$ is the longest element of $W_J$, we have $\mathscr{E}(w_J) = \Phi_J^+$ and since $t \notin J$ we have $\ell(tw_J) = \ell(w_J) + 1$ and  $w_J(\alpha_t) \notin \Phi_J^+$. Thus it follows from (iii) that $\mathscr{E}(tw_J)=\mathscr{E}(w_J)\cup\{w_J(\alpha_t)\}$, and therefore the elementary descent sets $\mathscr{E}(w_J)$ and $\mathscr{E}(tw_J)$ are distinct states of the automaton $\mathcal{A}_{BH}$. Hence it suffices to show that $T(tw_J) = T(w_J)$. If $v \in T(tw_J) $ then
\begin{equation*}
    \ell(tw_Jv) = \ell(tw_J) + \ell(v) = 1 + \ell(w_J) + \ell(v).
\end{equation*}
Hence $\ell(w_Jv) = \ell(w_J) + \ell(v)$ and thus $v \in T(w_J)$. To show the reverse inclusion, suppose that there is $v \in T(w_J) \setminus T(tw_J)$. Then we have $\ell(w_Jv) = \ell(w_J) + \ell(v)$ and 
\begin{equation*}
\ell(tw_Jv) < \ell(tw_J) + \ell(v) = 1 + \ell(w_J) + \ell(v) = 1 + \ell(w_Jv). 
\end{equation*}
Thus $\ell(tw_Jv) = \ell(w_Jv) - 1$ and therefore $t \in D(w_Jv)$. Since $w_J$ is the longest element of $W_J$ this implies that $J \cup \{ t \} \subseteq D(w_Jv)$ which is a contradiction with (ii) since left descent sets generate finite subgroups (see \cite[Corollary 2.18]{buildings}).
\end{proof}

\begin{proof}[Proof of \Cref{Theorem 1}]

(1) $\implies$ (2). We prove the contrapositive. Suppose that $\Gamma$ contains a subgraph $\Gamma'$ found in $\mathscr{X}$. Let $S'$ denote the vertices of $\Gamma'$. Our strategy is to choose $J \subset S'$ and $t \in S'$ such that conditions (i)-(iii) of \Cref{key_lemma} hold. 
\par
Suppose that $\Gamma'$ is the Coxeter graph of an irreducible, affine Weyl group (other than type $\tilde{A}_n$), and let $t$ be the unique node of $\Gamma'$ to which $s_0$ is connected in the standard numbering (see \cite{Bourbaki}). Let $\Phi_0 = \Phi_{S' \setminus \{s _0\}}$ and $W_0 = W_{S' \setminus \{s_0\}}$ with $w_0$ the longest element of $W_0$. After rescaling vectors we will assume that $\Phi_0$ is crystallographic, as outlined in \Cref{affine_root_system}. We let $J = S' \setminus \{t\}$. Since $\Gamma'$ is affine, it is clear that $|W_J| < \infty$ and $|W_{J \cup \{t\}}| = \infty$. We show that $w_J(\alpha_t) \in \mathscr{E}$.
\par
Let $J'=S'\backslash\{s_0,t\}$. Then $w_J=s_0w_{J'}$. We claim that $w_{J'}=w_0s_{\varphi}$. We prove the claim by showing that $\Phi(w_{J'})=\Phi_{J'}^+=\Phi(w_0s_{\varphi})$. The first equality is obvious. For the second equality, note that if $\alpha\in \Phi_{J'}^+$ then since $\langle \alpha,\varphi^{\vee}\rangle=0$ for all $\alpha\in J'$ we have $w_0s_{\varphi}(\alpha)=w_0(\alpha)\in -\Phi^+$, and so $\Phi_{J'}^+\subseteq \Phi(w_0s_{\varphi})$. Moreover, if $\beta\in \Phi_0^+\backslash\Phi_{J'}^+$ then $\beta=a\alpha_t+\gamma$ with $a\geq 1$ and $\gamma\in\Phi_{J'}^+$. Since $\langle\alpha_t,\varphi^{\vee}\rangle=1$ we compute $s_{\varphi}(\beta)=\beta-a\varphi\in -\Phi_0^+$ (since $\varphi$ is the highest root of $\Phi_0$) and hence $w_0s_{\varphi}(\beta)\in\Phi^+$. Therefore $\beta\notin\Phi(w_0s_{\varphi})$, and so $\Phi(w_0s_{\varphi})=\Phi_{J'}^+$ which completes the proof of the claim.
\par
Thus $w_J=s_0w_0s_{\varphi}$, and so
$$
w_J(\alpha_t)=s_0w_0s_{\varphi}(\alpha_t)=s_0w_0(\alpha_t-\varphi)=s_0(-\alpha_t+\varphi),
$$
where we have used the easily verified facts that $w_0\alpha_t=-\alpha_t$ and $w_0\varphi=-\varphi$. Since $s_0=s_{-\varphi+\delta}$ we compute $w_J(\alpha_t)=-\alpha_t+\delta$, 
and thus $w_J(\alpha_t)\in\mathscr{E}$. Hence the conditions of \Cref{key_lemma} hold and this completes the argument for the case of Coxeter graphs with a subgraph of affine type (other than type $\tilde{A}_n$).

We now consider the graphs $\Gamma'$ of the compact hyperbolic Coxeter groups which do not contain a circuit. Hence, for any choice of $t$, $W_J$ is finite for $J = S \setminus \{ t \}$ and $W_{J \cup \{ t \} }$ is infinite. Therefore by \Cref{key_lemma} we just need to show that there is a choice of $t$ such that $w_J(\alpha_t)$ is elementary. To do this we need to show that $w_J(\alpha_t)$ does not dominate any $\alpha\in\Phi_J^+$ and by \cite[Proposition 2.2(i)]{Brink} it suffices to show that $\langle w_J(\alpha_t), \alpha \rangle < 1$ for all $\alpha \in \Phi_J^+$. Since the bilinear form is $W$-invariant and $w_J(\Phi_J^+)=-\Phi_J^+$, this is equivalent to showing that $\langle \alpha_t,\alpha\rangle>-1$ for all $\alpha\in\Phi_J^+$. 
\par
Consider the graphs $X_3(a,b)$ with $a,b<\infty$ and $\frac{1}{a}+\frac{1}{b}<\frac{1}{2}$. Label the vertices $s_1, s_2, s_3$ reading from left to right and let $t =s_2$ and $J=\{s_1,s_3\}$. Thus \mbox{$\langle \alpha_t,\alpha\rangle>-1$} for all $\alpha\in\Phi_J^+=\{\alpha_1,\alpha_3\}$, and hence the result for these subgraphs.
\par
Now consider the graph $X_5(5)$. Label the vertices $s_1,\ldots,s_5$ reading left to right and choose $t = s_3$ and $J=\{s_1,s_2,s_4,s_5\}$. We have $\Phi_{J}^+ = \Phi_{\{s_1, s_2\}}^+ \sqcup \Phi_{\{s_4, s_5\}}^+$. Since $
\Phi_{\{s_1, s_2\}}^+ = \{ \alpha_1, \alpha_2, \alpha_1 + \zeta\alpha_2, \alpha_2 + \zeta\alpha_1, \zeta\alpha_1 + \zeta\alpha_2 \}$ where $\zeta = 2\cos\frac{\pi}{5}$
we have that
$
\langle \alpha_t, \alpha \rangle \in \{ \frac{1}{2}, 0, \frac{\zeta}{2} \}$ for all $\alpha \in \Phi_J^+$. Hence the result for this graph. All other cases are similar.
\\
\par
(2) $\implies$ (3). Let $\Gamma$ be a Coxeter graph which does not have a subgraph in~$\mathscr{X}$. As in \cite{Brink}, for $J \subseteq S$, let $\mathscr{E}_J = \{ \alpha \in \mathscr{E} \mid J(\alpha) = J \}$. Then $\mathscr{E}$ is the disjoint union of all $\mathscr{E}_J$ such that $\Gamma(J)$ is connected.
\par
Suppose that $\mathscr{E} \neq \Phi_{\mathrm{sph}}^+$. Then there exists $J \subset S$ such that $|W_J| = \infty$ and $\mathscr{E}_J \neq \emptyset$. By \Cref{no_circuits_no_infinitebonds}, the graph $\Gamma(J)$ cannot contain a circuit or an infinite bond. Hence $\Gamma(J)$ must be a tree with no infinite bonds. Since $\Gamma$ does not have any subgraph in $\mathscr{X}$, it follows that $\Gamma(J)$ does not have a subgraph in $\mathscr{X}$. \par
Let $m$ denote the maximal edge label in $\Gamma(J)$. Since $\Gamma(J)$ contains no subgraph of type $X_3(a,b)$ or $\tilde{G}_2$ we have $m<6$ (note that $|J|\geq 3$ as $W_J$ is infinite). 

Suppose that $m=5$, and let $e=\{s,t\}$ be an edge of $\Gamma(J)$ with edge label~$5$. Suppose that there is another edge $f=\{s',t'\}\neq e$ of $\Gamma(J)$ with edge label~$4$~or~$5$. Let $d$ be the distance between $e$ and $f$ (measured in the edge graph). The nonexistence of $X_3(4,5)$, $X_3(5,5)$, $X_j(4)$ and $X_j(5)$ subgraphs with $j=4,5$ forces $d> 3$, however then the nonexistence of $X_5(3)$ subgraphs gives a contradiction. So there is a unique bond with edge label $5$. Since $\Gamma(J)$ contains no $Y_4$, $X_5(3)$, $Z_4$ or~$Z_5$ graphs it follows that $\Gamma(J)$ is of type $H_3$ or $H_4$, a contradiction. 

Suppose that $m=4$. The nonexistence of $\tilde{C}_n$ subgraphs forces there to be a unique edge with label $4$. Since $\Gamma(J)$ contains no $\tilde{B}_n$ subgraphs the tree $\Gamma(J)$ has no branch points. Since $\Gamma(J)$ contains no $\tilde{F}_4$ subgraphs we see that $\Gamma(J)$ is of type $F_4$ or $B_n$, a contradiction.

Thus $m=3$. Since $\Gamma(J)$ contains no $\tilde{D}_4$ subgraph every vertex of $\Gamma(J)$ has degree at most $3$. Since $\Gamma(J)$ contains no $\tilde{D}_n$ subgraph with $n\geq 5$ there is at most one branch point. Suppose there is a branch point $x$. Thus $x$ is the unique vertex of $\Gamma(J)$ of degree $3$. Let $ p_i $ for $i \in \{0,1,2\}$ denote the three paths of $\Gamma(J)$ out of the vertex $x$, so that $\Gamma(J) = \cup_{i \in \{0,1,2\}} p_i$ and $\cap_{i \in \{0,1,2\}} p_i = \{x\}$. For $i\in\{0,1,2\}$ let $a_i$ denote the length of the path $p_i$, and without loss of generality we may suppose that $a_0\leq a_1\leq a_2$. The nonexistence of $\tilde{E}_6$ subgraphs implies that $a_0 = 1$. The nonexistence of $\tilde{E}_7$ subgraphs implies that $a_1\in\{1,2\}$. However if $a_1=1$ then $\Gamma(J)$ is of type $D_n$, and if $a_1=2$ then $\Gamma(J)$ is either of type $E_6$, $E_7$, or $E_8$, or it contains an $\tilde{E}_8$ subgraph, a contradiction. Thus $\Gamma(J)$ has no branch point and therefore must be of type $A_n$, a contradiction.
\\
\par
(3) $\implies$ (1). This implication was proven in \cite[Proposition 3.14]{Hohlweg}. For completeness, we repeat the argument here.
\par
Suppose that $T(w)=T(v)$. We are required to show that $\mathscr{E}(w)=\mathscr{E}(v)$. Suppose, for a contradiction, that there is $\alpha\in\mathscr{E}(w)\backslash\mathscr{E}(v)$, and let $J=J(\alpha)$. Let $w_1$ (respectively $v_1$) be the unique minimal length representative of the coset $wW_J$ (respectively $vW_J$). Then $w=w_1w_2$ and $v=v_1v_2$ for some $w_2,v_2\in W_J$, and moreover $\ell(w_1u)=\ell(w_1)+\ell(u)$ and $\ell(v_1u)=\ell(v_1)+\ell(u)$ for all $u\in W_J$ (see \cite[Proposition~2.20]{buildings}).

Let $z\in W_J$. Since $T(w)=T(v)$ we have $\ell(wz)=\ell(w)+\ell(z)$ if and only if $\ell(vz)=\ell(v)+\ell(z)$, and using the decompositions $w=w_1w_2$ and $v=v_1v_2$ it follows that $\ell(w_2z)=\ell(w_2)+\ell(z)$ if and only if $\ell(v_2z)=\ell(v_2)+\ell(z)$. Thus $T_J(w_2)=T_J(v_2)$, where we write $T_J(u)$ for the cone type in the group $W_J$ of an element $u\in W_J$. However since $\mathscr{E}=\Phi_{\mathrm{sph}}^+$ we have that $W_J$ is a finite group, and thus the associated Brink-Howlett automaton for this group is necessarily minimal (to see this, note firstly that the Brink-Howlett automaton has $|W_J|$ states because every positive root of a finite Coxeter group is elementary and the set of inversion sets is in bijection with elements of $W_J$, and secondly that the minimal automaton also has $|W_J|$ states because distinct elements of a finite Coxeter group have distinct cone types). Thus $\mathscr{E}(w_2)=\mathscr{E}(v_2)$. Now, by [13, Corollary~2.13] we have $\Phi(w_2)=\Phi(w)\cap \Phi_J^+$, and hence $\mathscr{E}(w_2)=\mathscr{E}(w)\cap \Phi_J^+$. Thus $\alpha\in \mathscr{E}(w_2)$. However since we also have $\mathscr{E}(v_2)=\mathscr{E}(v)\cap \Phi_J^+$ and $\mathscr{E}(w_2)=\mathscr{E}(v_2)$ we conclude that $\alpha\in \mathscr{E}(v)$, a contradiction. Thus $\mathscr{E}(w)=\mathscr{E}(v)$ as required.
\end{proof}

\begin{remark}
If $|S|=\infty$ then of course there is no finite state automaton recognising the language of reduced words of $(W,S)$. 
However in this setting the statement of Theorem~1, and its proof,  remain valid if one interprets ``minimal'' to mean that the map $\theta$ in \Cref{myhill_nerode} is bijective.
\end{remark}


\bibliographystyle{amsplain}

\end{document}